\documentclass[10pt]{article}
\usepackage[T1]{fontenc}
\usepackage[utf8]{inputenc}

\usepackage[english]{babel}
\usepackage{geometry}
\usepackage{extsizes}
\usepackage{amsmath}
\usepackage{amssymb}
\usepackage{amsthm}
\usepackage{subcaption}
\usepackage{graphicx}
\usepackage{xcolor}
\usepackage[colorlinks=true, allcolors=blue]{hyperref}
\usepackage[labelfont=bf, labelsep=space]{caption}

\usepackage{subcaption}  
\usepackage{float}

\usepackage[normalem]{ulem}

\newcommand{\ignore}[1]{}

\newtheoremstyle{style}
  {\topsep}   
  {\topsep}   
  {\upshape}  
  {}          
  {\bfseries}
  {.}       
  {.5em}     
  {}    

\theoremstyle{style}

\newtheorem{theorem}{Theorem}[section]
\newtheorem{lemma}[theorem]{Lemma}
\newtheorem{corollary}[theorem]{Corollary}
\newtheorem{proposition}[theorem]{Proposition}

\title{Eigen-componentwise convergence of SGD on quadratic programming}
\author{Lehan Chen and Yuji Nakatsukasa}
\date{December, 2024}
\begin{document}
\maketitle

\begin{abstract}
    \noindent
    Stochastic gradient descent (SGD) is a workhorse algorithm for solving large-scale optimization problems     in data science and machine learning. 
    Understanding the convergence of SGD is hence of fundamental importance. In this work we examine the SGD convergence (with various step sizes) when applied to unconstrained convex
    quadratic programming (essentially least-squares (LS) problems), and in particular analyze the error components respect to the eigenvectors of the Hessian. The main message is that the convergence depends largely on the corresponding eigenvalues (singular values of the coefficient matrix in the LS context), namely the components for the large singular values converge faster in the initial phase. We then show there is a phase transition in the convergence where the convergence speed of the components, especially those corresponding to the larger singular values, will decrease. Finally, we show that the convergence of the overall error (in the solution) tends to decay as more iterations are run, that is, the initial convergence is faster than the asymptote. 
\end{abstract}



\section{Introduction}
SGD is an optimization algorithm originally proposed by Robbins and Monroe~\cite{robbins1951stochastic} and is widely used in current machine learning practices~\cite{amiri2020machine, bottou2010large, shamir2013stochastic}.
In SGD, if we want to minimize the objective function taking the form \(F(x) = \frac{1}{N}\sum_{i=1}^{N} f_i(x)\), each time we choose \(i\) with a certain probability from \(1\ldots N\) and then update \(x_k\) by setting \(x_{k+1} = x_k-\alpha_k\nabla f_i(x_k)\), where \(\alpha_k\) is the step size, or learning rate, we use at the \(k\)-th step. Compared with the classical steepest descent, or gradient descent (GD) algorithm, SGD is more efficient per iteration. However, such advantage comes at a cost~\cite{bottou2018optimization}: when the objective function is strongly convex and smooth, using GD with fixed step size offers a linear convergence rate, while using SGD with fixed step size usually cannot guarantee convergence (in expectation) to the minimum. 
To do so, we need to set our step size to decay to 0, but that results in sub-linear convergence. Still, when the objective function is overly complicated, as happens for example in training neural networks, SGD is still useful because calculating individual \(\nabla f_i(x)\) is significantly easier than \(\nabla F(x)\). SGD has been studied extensively over the last decade~\cite{NEURIPS2018_a07c2f3b, pmlr-v89-li19c, 7875097, sclocchi2024different, smith2020generalization}. In this paper, to understand the behavior of SGD in terms of eigen-componentwise convergence (see below for a precise definition), we focus on a simple but nontrivial class of problems, namely linear least-squares (LS) problems \(\min_{x} \frac{1}{2} \|Ax-b\|^2\) where 
$A\in\mathbb{R}^{M\times N}$ and $M,N$ are large.

One approach for
 solving consistent least-squares problems is the randomized Kaczmarz algorithm~\cite{gower2015randomized, ma2022randomized, schopfer2019linear}, initially proposed by Kaczmarz~\cite{kaczmarz1937} in 1937. The basic idea is that for a least-squares problem \(\min_{x} \frac{1}{2}\|Ax-b\|^2\), we project \(x_k\) to the hyperplane \(\langle a_i, x\rangle = b_i\) each time, where \(a_i\) represents the \(i\)-th row of \(A\) and \(i\) cycles from 1 to \(M\), the number of rows \(A\) has. In subsequent work, Strohmer and Vershynin~\cite{strohmer2009randomized} propose a variant of randomized Kaczmarz algorithm and proved to have a linear convergence rate. Their idea is instead of choosing \(i\) cyclically, we choose row \(i\) with probability \(\|a_i\|^2/\|A\|_F^2\) at each step. Until now, Randomized Kaczmarz algorithm has been thoroughly investigated and improved~\cite{liu2016accelerated, needell2014paved, niu2020greedy}.

The motivation of our paper comes from Steinerberger~\cite{steinerberger2021randomized}, in which he shows (among other results) that 
when randomized Kaczmarz is applied to an LS problem \(\min_{x} \frac{1}{2} \|Ax-b\|^2\), 
the eigencomponentwise error
\(\langle x_k-x_*, v_\ell\rangle\) will converge significantly faster to 0 for smaller \(\ell\), 
where \(v_\ell\) denotes the (right) singular vector of \(A\) with corresponding singular value \(\sigma_\ell\). Equivalently, $v_{\ell}$ is the $\ell$-th eigenvector of the Hessian $A^\top A$.

The goal of this paper is to extend Steinerberger's result to SGD applied to general LS problems. More specifically, our analysis applies to general step sizes (and not speficially the one chosen by Randomized Kaczmarz), and LS problems that are not consistent; thereby getting the setting closer to the practical SGD usage.

 First, we introduce the eigen-componentwise convergence of (deterministic) GD on LS problems by giving an equality for \(\langle x_k-x_*, v_\ell\rangle\). Then we explore the eigen-componentwise convergence of SGD on LS problems with three different step sizes: fixed step size, \(1/k\) decaying step size, and \(1/k^\gamma\) decaying step size. In each of these cases, we first give an upper bound of \(\mathbb{E}[\langle x_k-x_*, v_\ell\rangle]\) to show that the components for large singular values converge faster in the initial phase, and we show that our predictions of \(\langle x_k-x_*, v_\ell\rangle\) is generally accurate in the initial phase. Then we give an upper bound of \(\mathbb{E}[\langle x_k-x_*, v_\ell\rangle^2]\) to illustrate the existence of phase transition after which the difference of the convergence speed of different components is no longer obvious. Finally, we use the two bounds to show that the error in the solution tends to decay as more iterations are run. 

Our results are important in several ways. First, both LS problems and SGD are common in a great variety of contexts and have been heavily explored for a long time~\cite{bjorck2024numerical, bottou1991stochastic, drineas2011faster, hardt2016train, shamir2016convergence}. Traditional randomized Kaczmarz algorithm does a good job at
solving LS problems, but only when the problem is consistent~\cite{bai2021greedy}. In our work, we give a detailed analysis applying SGD on LS problems with decaying step size, which can guarantee convergence to the optimal solution by bounding \(\mathbb{E}[\langle x_k-x_*, v_\ell\rangle]\). Also, by giving a bound of \(\mathbb{E}[\langle x_k-x_*, v_\ell\rangle^2]\), we show the existence of phase transition, which in turn explains the reason why the convergence rate of the overall error in the solution decreases as the number of iterations increases. That actually gives us another favorable property of SGD. In machine learning practices, we are not actually aiming for the minimal solution to the objective function as that could lead to various problems, like overfitting. So a common practice is choosing an optimization algorithm and iterates until the value of the objective function falls into a particular threshold. Our results show that the convergence speed of SGD is usually faster in the initial phase, therefore useful to practical purposes. 

Throughout the paper, we use \(\|\cdot\|\) and \(\|\cdot\|_F\) to denote the 2-norm and the Frobenius norm respectively.

\section{Motivation: Steinerberger's paper}
Steinerberger~\cite{steinerberger2021randomized} mainly includes three main results:
assuming the LS problem is consistent, that is, there exists an $x$ such that $Ax=b$,
\begin{equation} \label{Steinerberger1}
    \mathbb{E}[\langle x_k-x_*, v_\ell\rangle] = \left(1-\frac{\sigma_\ell^2}{\|A\|_F^2} \right)^k\langle x_0-x_*, v_\ell\rangle
\end{equation}
\begin{equation} \label{Steinerberger2}
    \mathbb{E}[\| x_{k+1}-x_*\|^2] \leq \left(1-\frac{1}{\|A\|_F^2}\left\|A\frac{x_k-x_*}{\|x_k-x_*\|}\right\|^2\right)\|x_k-x_*\|^2
\end{equation}
\begin{equation} \label{Steinerberger3}
    \mathbb{E}\left[\left\langle \frac{x_k-x_*}{\|x_k-x_*\|}, \frac{x_{k+1}-x_*}{\|x_{k+1}-x_*\|}\right\rangle^2\right] = 1-\frac{1}{\|A\|_F^2} \left\|A\frac{x_k-x_*}{\|x_k-x_*\|} \right\|^2
\end{equation}

\noindent
Equation~\eqref{Steinerberger1} tells us that \(\|\langle x_k-x_*, v_\ell\rangle\|^2\) converges faster for the largest singular vectors \(\ell\). Inequality~\eqref{Steinerberger2} illustrates that when \(x_k-x_*\) is mainly composed of singular vectors corresponding to small singular values, the convergence of \(\|x_k-x_*\|^2\) will be slower. Equation~\eqref{Steinerberger3} shows that once \(x_k\) is mainly composed of singular vectors corresponding to small singular values, then randomized Kaczmarz is less likely to explore other subspaces. Combining these results, we can give an interpretation of the behavior of \(\|x_k-x_*\|^2\) when using the (traditional) randomized Kaczmarz. Through numerical experiments, we observe that \(\|x_k-x_*\|^2\) will converge at a linear rate, but initially with higher convergence rate, then gradually decreasing.

It is reasonable to assume that initially \(x_0-x_*\) is composed of singular vectors corresponding to different singular values with roughly equal weights. Therefore, \(\|A(x_k-x_*)\|/\|x_k-x_*\|^2\) will be not too small in the first few iterations, that is, $O(\|A\|)$. But as the number of iterations progresses, since \(\langle x_k-x_*, v_\ell\rangle\) will converge faster to 0 for smaller \(\ell\), \(x_k-x_*\) will start to become mainly composed of singular vectors with small singular values, causing \(\|x_k-x_*\|^2\) to converge at a slower rate than initially. 

Figure~\ref{Figure1} represents the case where \(A \in \mathbb{R}^{300\times 150}, \sigma(A)\in [0.01, 1], l_1=1, l_2=135, l_3=150\).  First we notice that although both \(l_1\) and \(l_2\) have experienced phase transition, the blue curve and the green curve 
are still significantly below the red one. Before phase transition when the number of iteration is still small, we can indeed predict the behavior for \(l_1\) and \(l_2\) using Equation~\eqref{Steinerberger1}. But after that, their convergence rate will decrease significantly as more iterations are run. Similar results hold on the behavior of \(\|x_k-x_*\|^2\), represented by the yellow curve. These three important observations---faster convergence for smaller \(\ell\) in the initial stage, phase transition as iteration progresses, decreasing convergence speed of \(\|x_k-x_*\|^2\)---motivate us to give more rigorous interpretations on their behaviors and generalize them to wider SGD contexts.

\section{Eigencomponent Convergence of GD}
Here we study the eigencomponentwise convergence of the classic, deterministic gradient descent method. While the analysis is trivial, we are unaware of a reference that states the result, which clearly shows that the convergence is far from uniform. 

For \(M, N \in \mathbb{N}\) with \(M \geq N\), let \(A \in \mathbb{R}^{M \times N}\) and \(b \in \mathbb{R}^M\). Define the function \(F: \mathbb{R}^N \to \mathbb{R}\) by
\[F(x) = \frac{1}{2} \|Ax - b\|^2\]
Consider the following optimization problem
\begin{equation}
  \label{eq:minF}
\min_{x \in \mathbb{R}^N} F(x)  
\end{equation}
 First, we have the following elementary propositions and lemmas. 
\begin{proposition} \label{1}
     Denote \(x_* = \min_{x \in \mathbb{R}^N} F(x), H = A^\top A\). Then we have\[\forall x \in \mathbb{R}^N, \nabla F(x) = H(x-x_*)\]
\end{proposition}
\begin{proof}
    Notice that \(\nabla F(x) = Hx - A^\top b\), and as \(F(x)\) is convex, \(\nabla F(x_*) = Hx_* - A^\top b = 0\), and the result follows from direct calculation. 
\end{proof}
\noindent
The following result is from Steinerberger~\cite{steinerberger2021randomized}. Here we split it apart as a separate lemma.
\begin{lemma} \label{2}
    Let \(\sigma(A)\) denote the singular values of \(A\), and let \(\sigma_{\min}\) and \(\sigma_{\max}\) be the minimum and maximum singular value of \(A\). Let \(v_\ell\) be the right singular vector of \(A\) corresponding to the singular value \(\sigma_\ell\), hence, $Hv_i=\sigma_i^2v_i$. Then we have \[\forall x \in \mathbb{R}^N, \langle H(x-x_*), v_\ell\rangle = \sigma_\ell^2 \langle x-x_*, v_\ell\rangle\]
\end{lemma}
\begin{proof}
    We have \(Av_i = \sigma_i u_i\), where \(A = U\Sigma V^\top\), and
    \begin{align*}
        \langle H(x-x_*), v_\ell\rangle
        &= \langle A(x-x_*), Av_\ell\rangle \\
        &= \left\langle A\left(\sum_{i=1}^{N} \langle x-x_*, v_i\rangle v_i\right), Av_\ell\right\rangle \\
        &= \left\langle \sum_{i=1}^{N}\sigma_i \langle x-x_*, v_i\rangle u_i, \sigma_\ell u_\ell\right\rangle \\
        &= \sigma_\ell^2 \langle x-x_*, v_\ell\rangle
    \end{align*}
\end{proof}
\noindent
Now we implement gradient descent on \(F(x)\) with the update rule 
\[ x_{k+1} = x_k - \alpha_k \nabla F(x_k) \]
on the optimization problem \eqref{eq:minF}
\noindent
Then we have the following theorem, giving the value of \(\langle x_{n+1}-x_*, v_\ell\rangle\).
\begin{theorem} \label{3}
    If we apply gradient descent with step size \(\alpha_k\) at the \(k\)-th iteration, we have \[\langle x_{n+1}-x_*, v_\ell\rangle = \left(\prod_{k=0}^{n} (1-\alpha_k \sigma_\ell^2)\right) \langle x_0-x_*, v_\ell\rangle\]
\end{theorem}
\begin{proof}
    Using Proposition~\ref{1} and Lemma~\ref{2}, \(\forall 0 \leq k \leq n\), we have
    \begin{align*}
        \langle x_{k+1}-x_*, v_\ell \rangle
        &= \langle x_k-x_*, v_\ell\rangle - \alpha_k\langle\nabla F(x_k), v_\ell\rangle \\
        &= \langle x_k-x_*, v_\ell\rangle - \alpha_k\langle H(x_k-x_*), v_\ell\rangle \\
        &= (1-\alpha_k\sigma_\ell^2)\langle x_k-x_*, v_\ell\rangle
    \end{align*}
    Applying the relation recursively gives us the result. 
\end{proof}
\noindent
We call a function \(f(x)\) an \(L\)-smooth function if \(\forall x, y, \|\nabla f(x)-\nabla f(y)\| \leq L\|x-y\|\). Normally, to maximize the convergence speed when using gradient descent, for an \(L\)-smooth function, we will choose the step size \(\alpha = \frac{1}{L}\). In our setting, \(F(x)\) is \(\sigma_{\max}^2\)-smooth. Applying Theorem~\ref{3}, we have the following result.
\begin{corollary} \label{4}
    If we apply gradient descent with step size \(\alpha_k = \frac{1}{\sigma_{\max}^2}\) at the \(k\)-th iteration, we have
    \[\langle x_{n+1}-x_*, v_\ell\rangle = \left(1-\frac{\sigma_\ell^2}{\sigma_{\max}^2}\right)^{n+1} \langle x_0-x_*, v_\ell\rangle\]
\end{corollary}
\noindent
This results tells us that if we implement gradient descent on a quadratic function with the step size \(\alpha = \frac{1}{\sigma_{\max}^2}\), \(\langle x_k-x_*, v_\ell\rangle\) converges linearly for all \(\ell\). In particular, just after one single step, \(\langle x_1-x_*, v_1\rangle\) will become 0. That partially explains why the specific choice of step size \(\alpha = \frac{1}{L}\) is favorable.

In practice, when the size of the matrix \(A\) is too large, gradient descent is often impractical. A more commonly used approach is stochastic gradient descent. Now we explore the eigencomponent convergence of stochastic gradient descent by bounding \(\mathbb{E}[\langle x_n-x_*, v_\ell\rangle]\) and \(\mathbb{E}[\langle x_n-x_*, v_\ell\rangle ^2]\).

\section{Eigencomponent Convergence of SGD: General Results}
We now turn to the main subject of the paper: SGD.

Denote \(a_i\) as the \(i\)-th row of \(A \in \mathbb{R}^{M\times N}\), and \(b_i\) as the \(i\)-th entry of \(b\). For the optimization problem~\eqref{eq:minF}, if we use the following update rule 
\begin{equation} \label{eq: SGDupd}
    x_{k+1} = x_k + \alpha_k M \left( b_{i_k} - \langle a_{i_k}, x_k \rangle \right) a_{i_k}
\end{equation}
In this paper, we choose \(i_k\) uniformly and independently from \(1\) to \(M\), \(\forall k\). Now, we can interpret this process as using stochastic gradient descent on \( F(x) = \frac{1}{M} \sum_{i=1}^{M} f_i(x) \) where each \(f_i(x)\) is defined as \(\frac{M}{2} (a_i^\top x - b_i)^2\), with \(k\)-th step size \(\alpha_k\). The definition of \(f_i(x)\) explains the reason why~\eqref{eq: SGDupd} contains \(M\), the number of rows of \(A\). In this context, we have the following results.
\begin{theorem} \label{5}
    If we apply stochastic gradient descent with step size \(\alpha_k\) at \(k\)-th iteration, we have \[\mathbb{E}[\langle x_{n+1} - x_*, v_\ell \rangle | x_0] = \left(\prod_{k=0}^{n} \left( 1 - \alpha_k \sigma_\ell^2 \right)\right)  \langle x_{0} - x_*, v_\ell \rangle\]
\end{theorem}
\begin{proof}
    Notice that the update rule defined above gives us an unbiased estimation of \(\nabla F\). Therefore, we have
    \begin{align*}
        \mathbb{E}[\langle x_{k+1} - x_*, v_\ell \rangle | x_k] 
        &=\langle x_k - x_* - \sum_{i=1}^{M}\alpha_k \nabla f_i(x), v_\ell \rangle\\
        &=\langle x_k - x_*, v_\ell \rangle - \alpha_k \langle H(x_k - x_*), v_\ell \rangle\\
        &=(1 - \alpha_k \sigma_\ell^2) \langle x_k - x_*, v_\ell \rangle. 
    \end{align*}
    Applying the relation recursively gives us the result. 
\end{proof}
\noindent
Comparing to Theorem~\ref{3}, we see that in expectation, \(\mathbb{E}[\langle x_n-x_*, v_\ell \rangle]\) using SGD takes the same format as \(\langle x_n-x_*, v_\ell \rangle\) using GD. But in SGD, we also want to give an upper bound on \(\mathbb{E}[\langle x_n - x_*, v_\ell\rangle^2]\) to get bounds on the variance. To do so, we need the following results.
\begin{proposition} \label{6}
    For a general \(\mu\)-convex function \(f(x)\), if we denote \(f_* = \min_{x \in \mathbb{R}^N} f(x)\), we have
    \[2\mu(f(x)-f_*) \leq \|\nabla f(x)\|^2\]
\end{proposition}
\begin{proof}
    Since \(f\) is \(\mu\)-strongly convex, we have \(\forall x, y\):
    \[ f(y) \geq f(x) + \langle \nabla f(x), y-x\rangle + \frac{\mu}{2} \|y-x\|^2 \]
    By minimizing the right hand side with respect to \(y\), we find the minimizer is \(x - \frac{1}{\mu}\nabla f(x)\). 
    Therefore, we have \(\forall x, y\): \[ f(y) \geq f(x) - \frac{1}{2\mu} \|\nabla f(x)\|^2 \]
    The result then follows since this holds for all \(y\). 
\end{proof}
\begin{lemma} \label{7}
    Let \(\tilde{L} = \max_{1\leq i\leq m} \|a_i\|^2\), \(c(A) = \frac{\sigma_{\max}^2}{\sigma_{\min}^2}=(\kappa_2(A))^2\), and \(F_* = \min_{x \in \mathbb{R}^N} F(x)\), then after $k$ iterations of SGD, we have
    \[\mathbb{E} [\langle \nabla f_i(x_k), v_\ell \rangle ^2 | x_k] \leq  M\tilde{L}c(A)\sigma_{\max}^2\|x_k-x_*\|^2 + 2\alpha_k^2M\tilde{L}F_*\]
\end{lemma}
\begin{proof}
We have
    \begin{align*}
        \mathbb{E} [\langle \nabla f_i(x_k), v_\ell \rangle ^2 | x_k]
        &= \frac{1}{M} \sum_{i=1}^{M}\langle\nabla f_i(x_k), v_\ell\rangle ^2 \\
        &\leq \frac{1}{M} \sum_{i=1}^{M} \|\nabla f_i(x_k)\|^2 \\
        &= \frac{1}{M} \sum_{i=1}^{M} M^2 (\langle a_i, x_k\rangle - b_i)^2 \|a_i\|^2 \\
        &\leq M\tilde{L} \|Ax_k-b\|^2 \\
        &\leq \frac{M\tilde{L}}{\sigma_{\min}^2}\|H(x_k-x_*)\|^2 + 2M\tilde{L}F_* \\
        &\leq M\tilde{L}c(A)\sigma_{\max}^2\|x_k-x_*\|^2 + 2M\tilde{L}F_* 
    \end{align*}
    where the second to last inequality is obtained by Lemma~\ref{6} with the strongly convex parameter \(\mu = \sigma_{\min}^2\). 
\end{proof}
\noindent
Notice that the first inequality in the proof of Lemma~\ref{7} is loose and independent on \(\ell\), so there might be room for improvements. Still, our bound is enough to explain the asymptotic behavior of \(\langle x_k-x_*, v_\ell\rangle\) and illustrate the presence of phase transition.

The following lemma reveals the structure of the upper bound of \(\mathbb{E}[\langle x_k-x_*\rangle ^2]\) and gives the key recurrence relation that will be repetitively used later on.

\begin{lemma} \label{8}
    Define \(A(\alpha_k) = 1-2\alpha_k \sigma_\ell^2\), \(B(\alpha_k) = \alpha_k^2M\tilde{L}c(A)\sigma_{\max}^2\), and \(C(\alpha_k) = 2\alpha_k^2 M \tilde{L} F_*\), through the updates of SGD, we have
    \[\mathbb{E} [\langle x_{k+1} - x_*, v_\ell \rangle^2 | x_k] \leq A(\alpha_k) \langle x_k-x_*, v_\ell \rangle^2 + B(\alpha_k)\|x_k-x_*\|^2 + C(\alpha_k)\]
\end{lemma}
\begin{proof}
    \begin{align*}
        \mathbb{E} [\langle x_{k+1} - x_*, v_\ell \rangle^2 | x_k]
        &= \mathbb{E} [\langle x_k-x_* - \alpha_k \nabla f_i(x_k), v_\ell\rangle^2 | x_k] \\
        &= \langle x_k-x_*, v_\ell\rangle^2 - 2\alpha_k \langle x_k-x_*, v_\ell\rangle \mathbb{E} [\langle \nabla f_i(x_k), v_\ell \rangle | x_k] + \alpha_k^2 \mathbb{E} [\langle\nabla f_i(x_k), v_\ell\rangle^2 | x_k] \\
        &= \langle x_k-x_*, v_\ell\rangle^2 - 2\alpha_k \langle x_k-x_*, v_\ell\rangle \langle H(x_k-x_*), v_\ell \rangle + \alpha_k^2 \mathbb{E} [\langle\nabla f_i(x_k), v_\ell\rangle^2 | x_k ] \\
        &= \left(1-2\alpha_k \sigma_\ell^2\right) \langle x_k-x_*, v_\ell \rangle^2 + \alpha_k^2 \mathbb{E} [\langle\nabla f_i(x_k), v_\ell\rangle^2 | x_k] \\
        &\leq \left(1-2\alpha_k \sigma_\ell^2\right) \langle x_k-x_*, v_\ell \rangle^2 + \alpha_k^2M\tilde{L}c(A)\sigma_{\max}^2\|x_k-x_*\|^2 + 2\alpha_k^2 M\tilde{L}F_*
    \end{align*}
    Where the last inequality follows from Lemma~\ref{7}. 
\end{proof}
\noindent
Lemma~\ref{8} gives a recurrence relation on $\mathbb{E} [\langle x_{k+1} - x_*, v_\ell \rangle^2 | x_k]$
but \(\| x_k-x_*\|^2\) appears on the right hand side of the inequality. Its bound is given by the following lemma.
\begin{lemma} \label{9}
    If we choose step size \(\alpha_k\) such that \(\forall k \geq 0\), \(\alpha_k \leq \frac{1}{2M\tilde{L}}\), and define 
    \newline \(p(\alpha_k) = 1-\alpha_k\sigma_{\min}^2\), and \(q(\alpha_k) = 2\alpha_k^2\sigma^2\), then we have
    \[\mathbb{E}[\|x_{k+1} - x_*\|^2 | x_k] \leq p(\alpha_k)\|x_k-x_*\|^2 + q(\alpha_k)\]
\end{lemma}
\begin{proof}
    From Needell, Srebro, and Ward~\cite{needell2014stochastic} A.2, Proof of Theorem 2.1, if we define \(\sigma^2 = \mathbb{E}[\|\nabla f_i(x_*)\|^2]\), we have
    \[\mathbb{E}[\|x_{k+1} - x_*\|^2 | x_k] \leq (1 - 2\alpha_k \sigma_{\min}^2 (1 - \alpha_k M \tilde{L}))\|x_k - x_*\|^2 + 2 \alpha_k^2 \sigma^2\]
    If \(\alpha_k \leq \frac{1}{2M\tilde{L}}\), we have \(1 - 2\alpha_k \sigma_{\min}^2 (1 - \alpha_k M \tilde{L}) \leq 1-\alpha_k \sigma_{\min}^2\), completing the proof. 
\end{proof}
\noindent
Now we state the main theorem in this section, which upper bounds \(\mathbb{E}[\langle x_{n+1}-x_*, v_\ell\rangle ^2 | x_0]\). 
\begin{theorem} \label{10}
    \[\mathbb{E}[\langle x_{n+1}-x_*, v_\ell\rangle^2 | x_0] \leq A_0 \langle x_0-x_*, v_\ell\rangle^2 + B_0 \|x_0-x_*\|^2 + C_0\]
    Where \[A_0 = \prod_{i=0}^{n} A(\alpha_i)\] \[B_0 = \sum_{i=0}^{n} \tilde{B}_i\] \[C_0 = \left(\sum_{i=1}^{n} q(\alpha_{i-1})B_i\right) + \left(\sum_{i=1}^{n} A_iC(\alpha_{i-1})\right) + C(\alpha_n)\] \[\tilde{B}_i = A(\alpha_n)\cdots A(\alpha_{i+1})B(\alpha_{i})p(\alpha_{i-1})\cdots p(\alpha_0)\]
\end{theorem}
\noindent The proof of Theorem~\ref{10} is given via the lemma below.
\begin{lemma} \label{11}
    For fixed \(n \in \mathbb{N}\), \(\forall k \leq n\), we have
    \begin{equation}      \label{eq:induct}      
  \mathbb{E}[\langle x_{n+1}-x_*, v_\ell\rangle^2 | x_k] \leq A_k \langle x_k-x_*, v_\ell\rangle^2 + B_k \|x_k-x_*\|^2 + C_k
    \end{equation}  
    Where \[A_k = \prod_{i=k}^{n} A(\alpha_i)\] \[B_k = \sum_{i=0}^{n-k} \tilde{B}_i\] \[C_k = \left(\sum_{i=k+1}^{n} q(\alpha_{i-1})B_i\right) + \left(\sum_{i=k+1}^{n} A_i C(\alpha_{i-1})\right) + C(\alpha_n)\] \[\tilde{B}_i = A(\alpha_n)\cdots A(\alpha_{i+k+1})B(\alpha_{i+k})p(\alpha_{i+k-1})\cdots p(\alpha_k)\]
\end{lemma}
\begin{proof}
    The proof is based on induction. 
Certainly when \(k=n\), the claim holds. Now assume \[\mathbb{E}[\langle x_{n+1}-x_*, v_\ell\rangle^2 | x_{k+1}] \leq A_{k+1} \langle x_{k+1}-x_*, v_\ell\rangle^2 + B_{k+1} \|x_{k+1}-x_*\|^2 + C_{k+1}\] we have
    \begin{align*}
    \mathbb{E}[\langle x_{n+1}-x_*, v_\ell\rangle^2 | x_k]
    &= \mathbb{E}[\mathbb{E}[\langle x_{n+1}-x_*, v_\ell\rangle^2 | x_{k+1}, x_k] | x_k] \\
    &= \mathbb{E}[\mathbb{E}[\langle x_{n+1}-x_*, v_\ell\rangle^2 | x_{k+1}] | x_k] \\
    &\leq \mathbb{E}[ A_{k+1} \langle x_{k+1}-x_*, v_\ell\rangle^2 + B_{k+1} \|x_{k+1}-x_*\|^2 + C_{k+1} | x_k] \\
    &= A_{k+1} \mathbb{E}[\langle x_{k+1}-x_*, v_\ell\rangle^2 | x_k] + B_{k+1} \mathbb{E}[\|x_{k+1}-x_*\|^2 | x_k] + C_{k+1} \\
    &\leq A_{k+1}\left(A(\alpha_k)\langle x_k-x_*, v_\ell\rangle^2 + B(\alpha_k)\|x_k-x_*\|^2 + C(\alpha_k)\right) \\
    &\quad + B_{k+1}\left(p(\alpha_k)\|x_k-x_*\|^2+q(\alpha_k)\right) + C_{k+1} \\
    &= \left(A_{k+1}A(\alpha_k)\right)\langle x_k-x_*, v_\ell\rangle^2 + \left(A_{k+1}B(\alpha_k)+B_{k+1}p(\alpha_k)\right)\|x_k-x_*\|^2 \\
    &\quad + \left(A_{k+1}C(\alpha_k)+B_{k+1}q(\alpha_k)+C_{k+1}\right)
\end{align*}
Therefore, we see that~\eqref{eq:induct} holds by setting
\[A_k = A_{k+1}A(\alpha_k)\] \[B_k = A_{k+1}B(\alpha_k)+B_{k+1}p(\alpha_k)\] \[C_k = A_{k+1}C(\alpha_k)+B_{k+1}q(\alpha_k)+C_{k+1}\]
Now \(A_k\) follows easily from induction. For \(B_k\), assume
\[ B_{k+1} = \sum_{i=0}^{n-k-1} \tilde{B}_i \] where \[\tilde{B}_i = A(\alpha_n)\cdots A(\alpha_{i+k+2}) B(\alpha_{i+k+1}) p(\alpha_{i+k})\cdots p(\alpha_{k+1})\]
Therefore, we have
\begin{align*}
    A_{k+1}B(\alpha_k)+B_{k+1}p(\alpha_k)
    &= A(\alpha_n)\cdots A(\alpha_{k+1}) B(\alpha_k) \\
    &\quad + \sum_{i=0}^{n-k-1} A(\alpha_n)\cdots A(\alpha_{i+k+2}) B(\alpha_{i+k+1}) p(\alpha_{i+k})\cdots p(\alpha_k) \\
    &= \sum_{i=0}^{n-k} A(\alpha_n)\cdots A(\alpha_{i+k+1})B(\alpha_{i+k})p(\alpha_{i+k-1})\cdots p(\alpha_k) \\
    &= \sum_{i=0}^{n-k} \tilde{B}_i
\end{align*}
\\
For \(C_k\), assume
\[C_{k+1} = \left(\sum_{i=k+2}^{n} q(\alpha_{i-1})B_i\right) + \left(\sum_{i=k+2}^{n} A_i C(\alpha_{i-1})\right) + C(\alpha_n)\]
then we have
\begin{align*}
    C_k 
    &= A_{k+1}C(\alpha_k)+B_{k+1}q(\alpha_k)+C_{k+1} \\
    &= \left(\sum_{i=k+1}^{n} q(\alpha_{i-1})B_i\right) + \left(\sum_{i=k+1}^{n} A_i C(\alpha_{i-1})\right) + C(\alpha_n)
\end{align*} 
\end{proof}
\noindent
In the following sections we apply the theorems in two commonly chosen step sizes: fixed step size and polynomial decaying step size.
\section{Fixed Step Size}
We begin this section by giving an upper bound of \(\mathbb{E}[\langle x_n-x_*, v_\ell\rangle]\).
\begin{theorem} \label{12}
    From Theorem~\ref{3}, if we use SGD with fixed step size \(\alpha_k = \alpha\), we then have \[\mathbb{E}[\langle x_{n+1}-x_*, v_\ell\rangle] = (1-\alpha\sigma_\ell^2)^{n+1} \langle x_0-x_*, v_\ell\rangle\]
\end{theorem}
\noindent
However, Theorem~\ref{12} does not necessarily mean that \(x_n \to x_*\) as \(n \to \infty\). Actually we know things will behave differently depending on whether the LS problem is consistent or not. Therefore, to understand the convergence, we need to analyze \(\mathbb{E}[\langle x_n-x_*, v_\ell\rangle^2]\). However, if we simply apply 
Theorem~\ref{10} to aim to get the bound, we find that the bound will not converge to 0 as \(n \to \infty\). We need to further split the case into consistent problem and inconsistent problem.
\subsection{Consistent Problem}
For consistent problems, \(F_* = \min_{x \in \mathbb{R}^N} F(x) = 0\). The following theorem gives an upper bound of \(\mathbb{E}[\langle x_n-x_*, v_\ell\rangle^2]\), which, unfortunately, does not depend on \(\ell\).
\begin{theorem} \label{13}
    If we use fixed step size \(\alpha\) for a consistent problem, then we have \[\mathbb{E}[\langle x_{n+1}-x_*, v_\ell\rangle^2] \leq (1-2\alpha\sigma_{\min}^2+\alpha^2M\tilde{L}c(A)\sigma_{\max}^2)^{n+1} \|x_0-x_*\|^2\]
\end{theorem}
\begin{proof}
    First we know
    \begin{align*}
        \mathbb{E}[\|x_{k+1}-x_k\|^2]
        &= \frac{1}{M} \sum_{i=1}^{M} \|x_k-x_*-\alpha\nabla f_i(x_k)\|^2 \\
        &= \|x_k-x_*\|^2 - 2\alpha\langle x_k-x_*, \nabla F(x_k)\rangle + \alpha^2 \sum_{i=1}^{M} \|\nabla f_i(x_k)\|^2 \\
        &= \|x_k-x_*\|^2 - 2\alpha\langle x_k-x_*, H(x_k-x_*)\rangle + \alpha^2 \sum_{i=1}^{M} \|\nabla f_i(x_k)\|^2 \\
        &\leq \|x_k-x_*\|^2 - 2\alpha\sigma_{\min}^2 \|x_k-x_*\|^2 + \alpha^2 \sum_{i=1}^{M} \|\nabla f_i(x_k)\|^2
    \end{align*}
    In the proof of Lemma~\ref{7}, we have also shown that 
    \[ \frac{1}{M} \sum_{i=1}^{M} \|\nabla f_i(x_k)\|^2 \leq M\tilde{L}c(A)\sigma_{\max}^2\|x_k-x_*\|^2\]
    Therefore, we have
    \[ \mathbb{E}[\|x_{k+1}-x_*\|^2 | x_k] \leq (1-2\alpha\sigma_{\min}^2+\alpha^2M\tilde{L}c(A)\sigma_{\max}^2) \|x_k-x_*\|^2 \]
    Applying the above relation recursively, we have 
    \[\mathbb{E}[\|x_{n+1}-x_*\|^2] \leq (1-2\alpha\sigma_{\min}^2+\alpha^2M\tilde{L}c(A)\sigma_{\max}^2)^{n+1} \|x_0-x_*\|^2\]
    Finally, the result follows from noticing that \(\langle x-x_*, v_\ell\rangle^2 \leq \|x-x_*\|^2 \|v_\ell\|^2 = \|x-x_*\|^2\) 
\end{proof}
\noindent
Therefore, when \(\alpha\) is sufficiently small such that \(0 < (1-2\alpha\sigma_{\min}^2+\alpha^2M\tilde{L}c(A)\sigma_{\max}^2) < 1\), or equivalently, \(0 < \alpha < \frac{2}{M\tilde{L}c(A)^2}\), we can expect \(\mathbb{E}[\langle x_n-x_*, v_\ell\rangle^2]\) to converge with a linear rate. However, since the upper bound in Theorem~\ref{13} does not relate to \(\ell\) at all, we cannot deduce that \(\langle x_k-x_*, v_1\rangle\) will necessarily converge faster than \(\langle x_k-x_*, v_N\rangle\). Actually, using SGD with fixed step size on LS problem can be seen as a special instance of Randomized Kaczmarz when \(\forall i\), \(\|a_i\| = 1\) (\(a_i\) denotes the \(i\)-th row of \(A\)). So as Figure~\ref{Figure1} suggests, even though in the first stage of convergence, smaller \(\ell\) gives faster convergence speed, asymptotically as $k\to\infty$, their difference actually becomes subtle and negligible.

\subsection{Inconsistent Problem}
In the inconsistent case, we should not expect \(\mathbb{E}[\langle x_n-x_*, v_\ell\rangle^2]\) to converge to 0 anyway. The same applies to the (randomized or not) Kaczmarz method---which is arguably an important weakness. This is overcome by the step sizes we study next.

\section{\(1/k\) Decaying Step Size}
In this section, we assume the step size \(\alpha_k\) takes the form \(\alpha_k = \frac{a}{b+k}\) for \(a, b > 0\). First, as usual, we give an upper bound of \(\mathbb{E}[\langle x_k-x_*, v_\ell\rangle]\).
\begin{theorem} \label{14}
    If we use SGD with step size \(\alpha_k = \frac{a}{b+k}\), then we have\[|\mathbb{E}[\langle x_{n+1} - x_*, v_\ell \rangle]| \leq \left( \frac{b}{b+n} \right)^{a \sigma_\ell^2} |\langle x_0 - x_*, v_\ell \rangle|\]
\end{theorem}
\begin{proof}
    First, from Theorem~\ref{3}, we have 
    \[\mathbb{E}[\langle x_{n+1} - x_*, v_\ell \rangle] = \left(\prod_{k=0}^{n} \left( 1 - \frac{a \sigma_\ell^2}{b+k} \right)\right)\langle x_{0} - x_*, v_\ell \rangle\]
    We know \(\forall m \leq n \in \mathbb{N}, c>0, \prod_{k=m}^{n} \left(1-\frac{c}{k}\right) \leq \left(\frac{m}{n}\right)^c\) because 
    $$
    \log \left( \prod_{k=m}^{n} \left( 1 - \frac{c}{k} \right) \right) = \sum_{k=m}^{n} \log \left( 1 - \frac{c}{k} \right) \leq -\sum_{k=m}^{n} \frac{c}{k} \leq -c\int_{m}^{n} x^{-1} \mathrm{d}x = -c \log \left( \frac{n}{m} \right)
    $$
    Therefore, we know \[ \left(\prod_{k=0}^{n} \left( 1 - \frac{a \sigma_\ell^2}{b+k} \right)\right) \leq \left( \frac{b}{b+n} \right)^{a \sigma_\ell^2}\]
\end{proof}
\noindent
Now we start to focus on \(\mathbb{E}[\langle x_k-x_*, v_\ell\rangle^2]\). From Theorem~\ref{10}, we know that \[\mathbb{E}[\langle x_{n+1}-x_*, v_\ell\rangle^2 | x_0] \leq A_0\langle x_0-x_*, v_\ell\rangle^2 + B_0 \|x_0-x_*\|^2 + C_0\]
In the following parts, we aim to bound \(A_0, B_0, C_0\) respectively.
\begin{lemma} \label{14a}
    \(A_k \leq \left(\frac{b+k}{b+n}\right)^{2a\sigma_\ell^2}\) 
\end{lemma}
\begin{proof}
    Recall that 
    $$
    A_k = \prod_{i=k}^{n} A(\alpha_i) = \prod_{i=k}^{n} \left(1-\frac{2a\sigma_\ell^2}{b+k}\right) \leq \left(\frac{b+k}{b+n}\right)^{2a\sigma_\ell^2}
    $$
    Where the last inequality follows similarly from the proof of Theorem~\ref{14}, just letting \(c  = 2a\sigma_\ell^2\).
\end{proof}
\begin{lemma} \label{15}
    \(B_k \leq \frac{a^2M\tilde{L}c(A)\sigma_{\max}^2}{b+k-1} \left(\frac{b+k+1}{b+n+1}\right)^{a\sigma_{\min}^2}\)
\end{lemma}
\begin{proof}
    First, since \(\forall j, A(\alpha_j) \leq p(\alpha_j)\), and \(p(\alpha_j)\) increases with respect to \(j\), we have
    \begin{align*}
        \tilde{B}_i
        &\leq p(\alpha_n)\cdots p(\alpha_{i+k+1}) B(\alpha_{i+k}) p(\alpha_{i+k-1})\cdots p(\alpha_k) \\
        &\leq B(\alpha_{i+k}) p(\alpha_n)\cdots p(\alpha_{k+1})
    \end{align*}
    As \(B_k = \sum_{i=0}^{n-k} \tilde{B}_i\), we then have
    \begin{align*}
    B_k 
    &\leq \left(\sum_{i=k}^{n} B(\alpha_i)\right)\left(\prod_{i=k+1}^{n} p(\alpha_k)\right) \\
    &= a^2 M \tilde{L} c(A) \sigma_{\max}^2 \left(\sum_{i=k}^{n} \frac{1}{(b+i)^2}\right)\left(\prod_{i=k+1}^{n} 1-\frac{a\sigma_{\min}^2}{b+i}\right) \\
    &\leq a^2M\tilde{L}c(A)\sigma_{\max}^2 \left(\frac{1}{b+k-1}-\frac{1}{b+n}\right) \left(\frac{b+k+1}{b+n}\right)^{a\sigma_{\min}^2} \\
    &\leq \frac{a^2M\tilde{L}c(A)\sigma_{\max}^2}{b+k-1} \left(\frac{b+k+1}{b+n+1}\right)^{a\sigma_{\min}^2}
\end{align*}
\end{proof}
\noindent
For \(C_0\), we know \(C_0 \leq \left(\sum_{i=1}^{n} q(\alpha_{i-1})B_i\right) + \left(\sum_{i=1}^{n} A_iC(\alpha_{i-1})\right) + C(\alpha_n)\) from Theorem~\ref{10}. To bound \(C_0\), we need the following lemmas.

\begin{lemma} \label{16}
    If \(a\sigma_{\min}^2 <2\), then \(\sum_{k=1}^{n} q(\alpha_{i-1})B_i\) converges to 0 with the rate \(\left(\frac{1}{b+n}\right)^{a\sigma_{\min}^2}\).
    If \(a\sigma_{\min}^2 >2\), then \(\sum_{k=1}^{n} q(\alpha_{i-1})B_i\) converges to 0 with the rate \(\left(\frac{1}{b+n}\right)^2\). 
\end{lemma}
\begin{proof}
    We have
    \[ \sum_{k=1}^{n} q(\alpha_{i-1})B_i \leq 2a^4M\tilde{L}c(A)\sigma_{\max}^2\sigma^2 \sum_{k=1}^{n} \frac{1}{(b+k-1)^3} \left(\frac{b+k+1} {b+n+1}\right)^{a\sigma_{\min}^2} \]
    \begin{align*}
        \sum_{k=1}^{n} \frac{1}{(b+k-1)^3} \left(\frac{b+k+1} {b+n}\right)^{a\sigma_{\min}^2}
        &= \sum_{k=1}^{n} \left(\frac{b+k+1}{b+k-1}\right)^{3} \frac{1}{(b+k+1)^3} \left(\frac{b+k+1} {b+n+1}\right)^{a\sigma_{\min}^2} \\
        &\leq \left(\frac{b+1}{b-1}\right)^{3} \left(\sum_{k=1}^{n} (b+k+1)^{a\sigma_{\min}^2-3}\right) \frac{1}{(b+n+1)^{a\sigma_{\min}^2}}
    \end{align*}
    If \(a\sigma_{\min}^2 < 2\), then we have
    \begin{align*}
        \sum_{k=1}^{n} (b+k+1)^{a\sigma_{\min}^2-3} 
        &\leq \int_{b+1}^{b+n+1} x^{a\sigma_{\min}^2-3} \mathrm{d}x \\
        &\leq \frac{1}{2-a\sigma_{\min}^2} (b+1)^{a\sigma_{\min}^2-2}
    \end{align*}
    Therefore, we have
    \[ \sum_{k=1}^{n} q(\alpha_{i-1})B_i \leq 2a^4M\tilde{L}c(A)\sigma_{\max}^2\sigma^2 \left(\frac{b+1}{b-1}\right)^{3} \left(\frac{1}{2-a\sigma_{\min}^2}\right) \frac{(b+1)^{a\sigma_{\min}^2-2}}{(b+n+1)^{a\sigma_{\min}^2}} \]
    If \(a\sigma_{\min}^2 > 2\), then we have
    \begin{align*}
        \sum_{k=1}^{n} (b+k+1)^{a\sigma_{\min}^2-3}
        &\leq \int_{b+1}^{b+n+1} x^{a\sigma_{\min}^2-3} \mathrm{d}x \\
        &\leq \frac{1}{a\sigma_{\min}^2-2} (b+n+1)^{a\sigma_{\min}^2-2}
    \end{align*}
    Therefore, we have
    \[ \sum_{k=1}^{n} q(\alpha_{i-1})B_i \leq 2a^4M\tilde{L}c(A)\sigma_{\max}^2\sigma^2 \left(\frac{b+1}{b-1}\right)^{3} \left(\frac{1}{a\sigma_{\min}^2-2}\right) \left(\frac{1}{b+n+1}\right)^2\]
\end{proof}

\begin{lemma} \label{17}
    If \(a\sigma_\ell^2 < \frac{1}{2}\), then \(\sum_{i=1}^{n} A_iC(\alpha_{i-1})\) converges to 0 with the rate \(\left(\frac{1}{b+n}\right)^{2a\sigma_\ell^2}\). If \(a\sigma_\ell^2 > \frac{1}{2}\), \(\sum_{i=1}^{n} A_iC(\alpha_{i-1})\) converges to 0 with the rate \(\left(\frac{1}{b+n}\right)\).
\end{lemma}
\begin{proof}
    We have
    \begin{align*}
        \sum_{i=1}^{n} A_iC(\alpha_{i-1}) 
        &\leq 2a^2M\tilde{L}F_* \sum_{k=1}^{n} \frac{1}{(b+k-1)^2} \left(\frac{b+k}{b+n}\right)^{2a\sigma_\ell^2} \\
        &\leq \frac{2a^2M\tilde{L}F_*}{(b+n)^{2a\sigma_\ell^2}}\left(\frac{b}{b-1}\right)^2 \sum_{k=1}^{n} (b+k)^{2a\sigma_\ell^2-2}
    \end{align*}
    If \(a\sigma_\ell^2 < \frac{1}{2}\), then we have
    $$
        \sum_{k=1}^{n} (b+k)^{2a\sigma_\ell^2-2} \leq \int_{b}^{b+n} x^{2a\sigma_\ell^2-2} \mathrm{d}x \leq \frac{b^{2a\sigma_\ell^2-1}}{1-2a\sigma_\ell^2}
    $$
    Therefore, we have
    \[ \sum_{i=1}^{n} A_iC(\alpha_{i-1}) \leq 2a^2M\tilde{L}F_*\left(\frac{b}{b-1}\right)^2 \left( \frac{b^{2a\sigma_\ell^2-1}}{1-2a\sigma_\ell^2} \right) \left(\frac{1}{b+n}\right)^{2a\sigma_\ell^2}\]
    If \(a\sigma_\ell^2 > \frac{1}{2}\), then we have
    $$
        \sum_{k=1}^{n} (b+k)^{2a\sigma_\ell^2-2} \leq \int_{b}^{b+n} x^{2a\sigma_\ell^2-2} \mathrm{d}x \leq \left(\frac{1}{2a\sigma_\ell^2-1}\right) (b+n)^{2a\sigma_\ell^2-1}
    $$
    Therefore, we have
    \[ \sum_{i=1}^{n} A_iC(\alpha_{i-1}) \leq \left(\frac{2a^2M\tilde{L}F_*}{2a\sigma_\ell^2-1}\right) \left(\frac{b}{b-1}\right)^2 \left(\frac{1}{b+n}\right)\]
\end{proof}
\noindent
Recall that in Theorem~\ref{10}, we know that \(\mathbb{E}[\langle x_{n+1}-x_*, v_\ell\rangle^2 | x_0] \leq A_0\langle x_0-x_*, v_\ell\rangle^2 + B_0 \|x_0-x_*\|^2 + C_0\). Having already derived the upper bounds of \(A_0, B_0, C_0\), now we can summarize the convergence speed of \(\mathbb{E}[\langle x_k-x_*, v_\ell\rangle^2]\).

\begin{theorem} \label{18}
    \(A_0\) converges to 0 with the rate \(\left(\frac{b}{b+n}\right)^{2a\sigma_\ell^2}\), \(B_0\) converges to 0 with the rate \(\left(\frac{b}{b+n}\right)^{a\sigma_{\min}^2}\), and \(C_0\) converges to 0 with the rate \(\left(\frac{b}{b+n}\right)^{\eta}\) for some \(\eta \in (0, 2].\) (The convergence speed for \(C_0\) is not that easy to summarize, and here \(\eta\) is partially dependent on \(\ell\))
\end{theorem}
\noindent
Since we have not made any assumption whether the given problem is consistent or not, we can ensure the randomized Kaczmarz algorithm with diminishing step size will converge to \(x_*\) in the inconsistent case, while the classical randomized Kaczmarz algorithm cannot.

Notice there is a natural trade off between the convergence rate of \(|\langle x_k-x_*, v_\ell\rangle|\) and \(\mathbb{V} \langle x_k-x_*, v_\ell \rangle\): if \(a\) is small, then though we will converge in a slower rate, the variance can be significantly smaller than if \(a\) is large.

Through the numerical experiments, we have also noticed that in general, as the number of iterations grows, the convergence speed of \(\langle x_k-x_*, v_\ell\rangle\) will decrease, especially for small \(\ell\), which is not surprising considering the upper bound we get for \(\mathbb{E}[\langle x_k-x_*, v_\ell\rangle^2]\): Both \(B_0\) and \(C_0\) are dependent on \(\sigma_{\min}^2\) to some extent, so as the number of iterations progresses, \(B_0\) and \(C_0\) will become dominant. Still, there is an \(a^2\) in \(B_0\) and \(C_0\). So when \(a\) is small, we can still expect that \(\langle x_k-x_*, v_\ell\rangle\) converges faster for smaller \(\ell\), at least when the number of iterations is not too large. 

Figure~\ref{Figure2} represents the case where \(A \in \mathbb{R}^{30 \times 20}\), \(\sigma(A) \in [0.1, 1]\), \(a=0.5, b=20\), \(l_1=1, l_2=10, l_3=20\), Iter=10000, Rept=20 for a consistent problem. And Figure~\ref{Figure3} represents the case for a general (inconsistent) problem with the same configuration. Notice that because Iter is relatively small, our predictions on \(\langle x_k-x_*, v_\ell\rangle\) are in general accurate, both for consistent and inconsistent case. Since \(\left(v_\ell\right)_{l=1}^{N}\) form an orthonormal basis for \(\mathbb{R}^{N}\), we have 
\begin{proposition} \label{25}
    \[ \mathbb{E}[\|x_k-x_*\| ^2] = \sum_{l=1}^{N} \mathbb{E}[\langle x_k-x_*, v_\ell\rangle ^2] \]
\end{proposition}
\noindent
Theorem~\ref{18} tells us that as iteration progresses, $B_0$ and $C_0$ will become dominant, and the convergence of $\langle x_k-x_*, v_\ell$ will slow down for small $\ell$. Therefore, we should expect that \(\|x_k-x_*\|^2\) will experience phase transition: its convergence speed will decrease as the number of iterations increases. And our numerical experiments support us for the claim. Figure~\ref{Figure4} represents the case where \(A \in \mathbb{R}^{10000 \times 3000}\), \(\sigma(A) \in [0.01, 1]\), \(a=0.5, b=150\), Iter=\(10^6\), Rept=1 for a consistent problem. Note that the convergence speed starts to decrease at around 30000 iterations.

Theorem~\ref{5} tells us that we can expect faster convergence of \(\mathbb{E}[\langle x_k-x_*, v_\ell\rangle]\) if step sizes are larger. So naturally we want to explore other step sizes that decay slower. That leads to the discussion in the next section.
\section{\(1/k^\gamma\) Decaying Step Size}
Now we consider the step size which takes the form \(\alpha_k = \frac{a}{(b+k)^\gamma}\). By Robbins-Monro conditions~\cite{robbins1951stochastic}, we only need to consider the case when \(\frac{1}{2}<\gamma<1\). As usual, we first give an upper bound of \(\mathbb{E}[\langle x_k-x_*, v_\ell\rangle]\).

\begin{theorem} \label{19}
    If we use SGD with step size \(\alpha_k = \frac{a}{(b+k)^\gamma}\) for \(1/2 \leq \gamma \leq 1\), then we have
    \[|\mathbb{E}[\langle x_{n+1}-x_*, v_\ell\rangle]| \leq e^{\left(\frac{a\sigma_\ell^2}{1-\gamma}b^{1-\gamma}\right)}e^{-\left(\frac{a\sigma_\ell^2}{1-\gamma}(b+n)^{1-\gamma}\right)}|\langle x_0-x_*, v_\ell\rangle|\]
\end{theorem}
\begin{proof}
    From Theorem~\ref{5}, we have
    \[\mathbb{E}[\langle x_{n+1}-x_*, v_\ell\rangle] = \prod_{k=0}^{n} \left(1-\frac{a\sigma_\ell^2}{(b+k)^\gamma}\right) \langle x_0-x_*, v_\ell\rangle\]
    \begin{align*}
        \log \left(\prod_{k=0}^{n} \left(1-\frac{a\sigma_\ell^2}{(b+k)^\gamma}\right)\right)
        &= \sum_{k=0}^{n} \log\left(1-\frac{a\sigma_\ell^2}{(b+k)^\gamma}\right) \\
        &\leq -a\sigma_\ell^2\sum_{k=b}^{b+n} \frac{1}{k^\gamma} \\
        &\leq -a\sigma_\ell^2\int_{m}^{n} \frac{1}{x^\gamma} \mathrm{d}x \\
        &\leq -\frac{a\sigma_\ell^2}{1-\gamma}\left((b+n)^{1-\gamma}-b^{1-\gamma}\right)
    \end{align*}
    Therefore, taking the absolute value on both sides, we have
    \[ |\mathbb{E}[\langle x_{n+1}-x_*, v_\ell\rangle]| \leq e^{\left(\frac{a\sigma_\ell^2}{1-\gamma}b^{-\gamma+1}\right)}e^{-\left(\frac{a\sigma_\ell^2}{1-\gamma}(b+n)^{-\gamma+1}\right)}|\langle x_0-x_*, v_\ell\rangle| \]
\end{proof}
\noindent
As a result, we know that for \(\gamma < 1\), \(\mathbb{E}[\langle x_n-x_*, v_\ell\rangle]\) converges to 0 with a root exponential convergence rate. Now we start to focus on \(\mathbb{E}[\langle x_k-x_*, v_\ell\rangle^2]\). From Theorem~\ref{10}, we know that \[\mathbb{E}[\langle x_{n+1}-x_*, v_\ell\rangle^2 | x_0] \leq A_0\langle x_0-x_*, v_\ell\rangle^2 + B_0 \|x_0-x_*\|^2 + C_0\]
We give the bounds of \(A_0, B_0, C_0\) respectively in the following lemmas.

\begin{lemma} \label{20}
    \(A_k \leq e^{\left(\frac{2a\sigma_\ell^2}{1-\gamma}(b+k)^{1-\gamma}\right)}e^{-\left(\frac{2a\sigma_\ell^2}{1-\gamma}(b+n)^{1-\gamma}\right)}\)
\end{lemma}
\begin{proof}
    Recall that 
    $$
    A_k = \prod_{i=k}^{n} A(\alpha_i) = \prod_{i=k}^{n} \left(1-\frac{2a\sigma_\ell^2}{(b+k)^\gamma}\right) \leq e^{\left(\frac{2a\sigma_\ell^2}{1-\gamma}(b+k)^{1-\gamma}\right)}e^{-\left(\frac{2a\sigma_\ell^2}{1-\gamma}(b+n)^{1-\gamma}\right)}
    $$
    Where the last inequality follows similarly from the proof of Lemma~\ref{19}. \\
\end{proof}

\begin{lemma} \label{21}
    \(B_k \leq \frac{a^2M\tilde{L}c(A)\sigma_{\max}^2}{2\gamma-1}e^{\left(\frac{a\sigma_{\min}^2}{1-\gamma}(b+k+1)^{1-\gamma}\right)}(b+k-1)^{1-2\gamma}e^{-\left(\frac{a\sigma_{\min}^2}{1-\gamma}(b+n+1)^{1-\gamma}\right)}\)
\end{lemma}
\begin{proof}
     Since the step size \(\alpha_k\) is decreasing, similar to the proof of Lemma \ref{15}, we still have \(B_k \leq \left(\sum_{i=k}^{n} B(\alpha_i)\right)\left(\prod_{i=k+1}^{n} p(\alpha_i)\right)\).
     For \(\sum_{i=k}^{n} B(\alpha_i)\), we have
     \begin{align*}
         \sum_{i=k}^{n} B(\alpha_i) 
        &= \alpha^2M\tilde{L}c(A)\sigma_{\max}^2\sum_{i=b+k}^{b+n}\frac{1}{i^{2\gamma}} \\
        &\leq \alpha^2M\tilde{L}c(A)\sigma_{\max}^2 \int_{b+k-1}^{b+n} \frac{1}{x^{2\gamma}} \mathrm{d}x \\
        &\leq \frac{a^2M\tilde{L}c(A)\sigma_{\max}^2}{2\gamma-1}(b+k-1)^{1-2\gamma}
     \end{align*}
     Similar to the proof of Theorem~\ref{20}, we have \[ \prod_{i=k+1}^{n} p(\alpha_i) \leq e^{\left(\frac{a\sigma_{\min}^2}{1-\gamma}(b+k+1)^{1-\gamma}\right)} e^{-\left(\frac{a\sigma_{\min}^2}{1-\gamma}(b+n+1)^{1-\gamma}\right)} \]
     
\end{proof}
\noindent
Therefore we know \(B_0\) converges to 0 with a root exponential convergence rate independent on \(\sigma_\ell^2\). 

For \(C_0\), we know \(C_0 \leq \left(\sum_{i=1}^{n} q(\alpha_{i-1})B_i\right) + \left(\sum_{i=1}^{n} A_iC(\alpha_{i-1})\right) + C(\alpha_n)\) from Theorem~\ref{10}. To bound \(C_0\), we have the following lemmas.

\begin{lemma} \label{22}
    When \(n\) sufficiently large, we have \[\sum_{i=1}^{n} q(\alpha_{i-1})B_i \leq \frac{a^4M\tilde{L}c(A)\sigma_{\max}^2\sigma^2}{2(2\gamma-1)^2}\left(\frac{b-1}{b+1}\right)^{1-4\gamma} \left(e^{-\frac{a\sigma_{\min}^2}{1-\gamma}\left(1-\left(\frac{1}{2}\right)^{1-\gamma}\right)(b+n+1)^{1-\gamma}} (b+1)^{2-4\gamma} + \left(\frac{b+n+1}{2}\right)^{2-4\gamma}\right)\]
\end{lemma}
\begin{proof}
    Without loss of generality, assume \(n/2 \in \mathbb{N}\). Therefore $\forall c \in \mathbb{R}$, we have
    \begin{align*}
        e^{-cn^{1-\gamma}} \sum_{k=m}^{n} e^{ck^{1-\gamma}}k^{1-4\gamma}
        &= e^{-cn^{1-\gamma}}\sum_{k=m}^{n/2} e^{ck^{1-\gamma}}k^{1-4\gamma}+ e^{-cn^{1-\gamma}}\sum_{k=n/2+1}^{n}e^{ck^{1-\gamma}}k^{1-4\gamma} \\
        &\leq e^{-c\left(n^{1-\gamma}-(n/2)^{1-\gamma}\right)}\sum_{k=m}^{n/2}k^{1-4\gamma} + \sum_{k=n/2+1}^{n}k^{1-4\gamma} \\
        &\leq e^{-c\left(n^{1-\gamma}-(n/2)^{1-\gamma}\right)} \frac{1}{4\gamma-2}(m-1)^{2-4\gamma} + \frac{1}{4\gamma-2}\left(\frac{n}{2}\right)^{2-4\gamma} \\
        &= e^{-c\left(1-\left(\frac{1}{2}\right)^{1-\gamma}\right)n^{1-\gamma}} \frac{1}{4\gamma-2}(m-1)^{2-4\gamma} + \frac{1}{4\gamma-2}\left(\frac{n}{2}\right)^{2-4\gamma}
    \end{align*}
    Letting \(c=\frac{a\sigma_{\min}^2}{1-\gamma}\), we then have
    \begin{align*}
        \sum_{i=1}^{n} q(\alpha_{i-1})B_i
        &\leq \frac{a^4M\tilde{L}c(A)\sigma_{\max}^2\sigma^2}{2\gamma-1} \left(\frac{b-1}{b+1}\right)^{1-4\gamma} e^{-c(b+n+1)^{1-\gamma}} \sum_{k=b+2}^{b+n+1} e^{ck^{1-\gamma}}k^{1-4\gamma} \\
        &\leq \frac{a^4M\tilde{L}c(A)\sigma_{\max}^2\sigma^2}{2(2\gamma-1)^2} \left(\frac{b-1}{b+1}\right)^{1-4\gamma} \left(e^{-\frac{a\sigma_{\min}^2}{1-\gamma}\left(1-\left(\frac{1}{2}\right)^{1-\gamma}\right)(b+n+1)^{1-\gamma}} (b+1)^{2-4\gamma} + \left(\frac{b+n+1}{2}\right)^{2-4\gamma}\right) \\
    \end{align*} 
\end{proof}

\begin{lemma} \label{23}
    When \(n\) sufficiently large, we have
    \[ \sum_{i=1}^{n} A_iC(\alpha_{i-1}) \leq \frac{2a^2M\tilde{L}F_*\sigma^2}{2\gamma-1}\left(\frac{b-1}{b}\right)^{-2\gamma} \left(e^{-\frac{2a\sigma_\ell^2}{1-\gamma}\left(1-\left(\frac{1}{2}\right)^{1-\gamma}\right)n^{1-\gamma}} b^{1-2\gamma} + \left(\frac{b+n}{2}\right)^{1-2\gamma}\right)\]
\end{lemma}
\begin{proof}
    The idea is almost the same as the proof of Lemma~\ref{22}. Without loss of generality, assume \(n/2 \in \mathbb{N}\). Therefore $\forall c \in \mathbb{R}$ we have
    \begin{align*}
        e^{-cn^{1-\gamma}} \sum_{k=m}^{n} e^{ck^{1-\gamma}}k^{-2\gamma}
        &= e^{-cn^{1-\gamma}}\sum_{k=m}^{n/2} e^{ck^{1-\gamma}}k^{-2\gamma}+ e^{-cn^{1-\gamma}}\sum_{k=n/2+1}^{n}e^{ck^{1-\gamma}}k^{-2\gamma} \\
        &\leq e^{-c\left(n^{1-\gamma}-(n/2)^{1-\gamma}\right)}\sum_{k=m}^{n/2}k^{-2\gamma} + \sum_{k=n/2+1}^{n}k^{-2\gamma} \\
        &\leq e^{-c\left(n^{1-\gamma}-(n/2)^{1-\gamma}\right)} \frac{1}{2\gamma-1}(m-1)^{1-2\gamma} + \frac{1}{2\gamma-1}\left(\frac{n}{2}\right)^{1-2\gamma} \\
        &= e^{-c\left(1-\left(\frac{1}{2}\right)^{1-\gamma}\right)n^{1-\gamma}} \frac{1}{2\gamma-1}(m-1)^{1-2\gamma} + \frac{1}{2\gamma-1}\left(\frac{n}{2}\right)^{1-2\gamma}
    \end{align*}
    Letting \(c=\frac{2a\sigma_\ell^2}{1-\gamma}\), we then have
    \begin{align*}
        \sum_{i=1}^{n} A_iC(\alpha_{i-1})
        &\leq 2a^2M\tilde{L}F_*\sigma^2 \left(\frac{b-1}{b}\right)^{-2\gamma} e^{-\frac{2a\sigma_\ell^2}{1-\gamma}(b+n)^{1-\gamma}}\sum_{k=b+1}^{b+n} e^{\frac{2a\sigma_\ell^2}{1-\gamma}k^{1-\gamma}} k^{-2\gamma} \\
        &\leq \frac{2a^2M\tilde{L}F_*\sigma^2}{2\gamma-1} \left(\frac{b-1}{b}\right)^{-2\gamma} \left(e^{-\frac{2a\sigma_\ell^2}{1-\gamma}\left(1-\left(\frac{1}{2}\right)^{1-\gamma}\right)n^{1-\gamma}} b^{1-2\gamma} + \left(\frac{b+n}{2}\right)^{1-2\gamma}\right) \\
    \end{align*}
\end{proof}
\noindent
Recall that in Theorem~\ref{10}, we know that \(\mathbb{E}[\langle x_{n+1}-x_*, v_\ell\rangle^2 | x_0] \leq A_0\langle x_0-x_*, v_\ell\rangle^2 + B_0 \|x_0-x_*\|^2 + C_0\). Having already derived the upper bounds of \(A_0, B_0, C_0\), now we can summarize the convergence speed of \(\mathbb{E}[\langle x_k-x_*, v_\ell\rangle^2]\).

\begin{theorem} \label{24}
    \(A_0\) converges to 0 with the rate \(e^{-\left(\frac{2a\sigma_\ell^2}{1-\gamma}(b+n)^{1-\gamma}\right)}\), \(B_0\) converges to 0 with the rate \(e^{-\left(\frac{a\sigma_{\min}^2}{1-\gamma}(b+n)^{1-\gamma}\right)}\), and \(C_0\) converges to 0 with the rate \(\left(\frac{b}{b+n}\right)^{\eta}\) for some \(\eta \in (0, 1]\).
\end{theorem}
\noindent
In this setting, when \(a\) is small and the iterations is not too large, the prediction is still relatively accurate. The reason is almost the same as the setting in the section on the \(1/k\) decaying step size. Figure~\ref{Figure5} illustrates the case where \(A \in \mathbb{R}^{30 \times 20}\), \(\sigma(A) \in [0.1, 1]\), \(a=0.2, b=5, \gamma=0.8\), \(l_1=1, l_2=10, l_3=20\), Iter=10000, Rept=20 for a consistent problem. Notice that the light purple curve represents an estimation of \(\langle x_k-x_*, v_\ell\rangle\) where \(\gamma=1\). The difference between the light purple curve and light blue curve is striking. Figure~\ref{Figure6} represents the case with the same setting for a general (inconsistent) problem. 

When the number of iterations is too large, from the numerical experiments, the convergence speed of \(\langle x_k-x_*, v_\ell\rangle\) will decrease, especially for smaller \(\ell\). Actually, on a log-log scale plot, \(\langle x_k-x_*, v_\ell\rangle\) will start to become a straight line. This phenomenon is still within our expectation by considering the convergence speed of \(C_0\) in Theorem~\ref{24}. Using the above settings for a consistent problem, Figure~\ref{Figure7} shows that after 10000 iterations, phase transition appears on \(\langle x_k-x_*, v_1\rangle\), and its convergence speed no more increases, implying that it no longer converges in a root exponential way. Therefore, we should not expect the convergence rate differ by a lot between \(1/k^\gamma\) step size and \(1/k\) step size when the number of iterations is large. But still, \(1/k^\gamma\) step size will converge way faster in the initial stage, as suggested above. This result can also be verified by the behavior of \(\|x_k-x_*\|^2\). Figure~\ref{Figure8} illustrates the case where \(A \in \mathbb{R}^{10000 \times 2000}\), \(\sigma(A) \in [0.01, 1]\), \(a=0.06, b=50, \gamma=0.7\), Iter=\(10^6\), Rept=1 for a consistent problem. Notice that after around 1000 iterations, the yellow curve starts to become a straight line. 
\newpage

\section{Numerical Experiments and Comments}
In the following figures, the legend is used as follows:
\begin{enumerate}
    \item blue solid line with `o' markers: \(\langle x_k-x_*, v_{\ell_1}\rangle / \langle x_0-x_*, v_{\ell_1}\rangle\)
    \item light blue solid line: its prediction
    \item green dash line with `+' markers: \(\langle x_k-x_*, v_{\ell_2}\rangle / \langle x_0-x_*, v_{\ell_2}\rangle\)
    \item light green dash line: its prediction
    \item red dash dot line with `*' markers: \(\langle x_k-x_*, v_{\ell_3}\rangle / \langle x_0-x_*, v_{\ell_3}\rangle\)
    \item light red dash dot line: its prediction
    \item yellow solid line: \(\|x_k-x_*\|^2 / \|x_0-x_*\|^2\)
    \item light purple solid line: rediction of \(\langle x_k-x_*, v_{\ell_1}\rangle\ / \langle x_0-x_*, v_{\ell_1}\rangle\) with \(\gamma=1\), but actually \(\gamma < 1\)
\end{enumerate}

\begin{figure}[h]
    \centering
    \begin{subfigure}{0.45\textwidth}
        \centering
        \includegraphics[width=\textwidth]{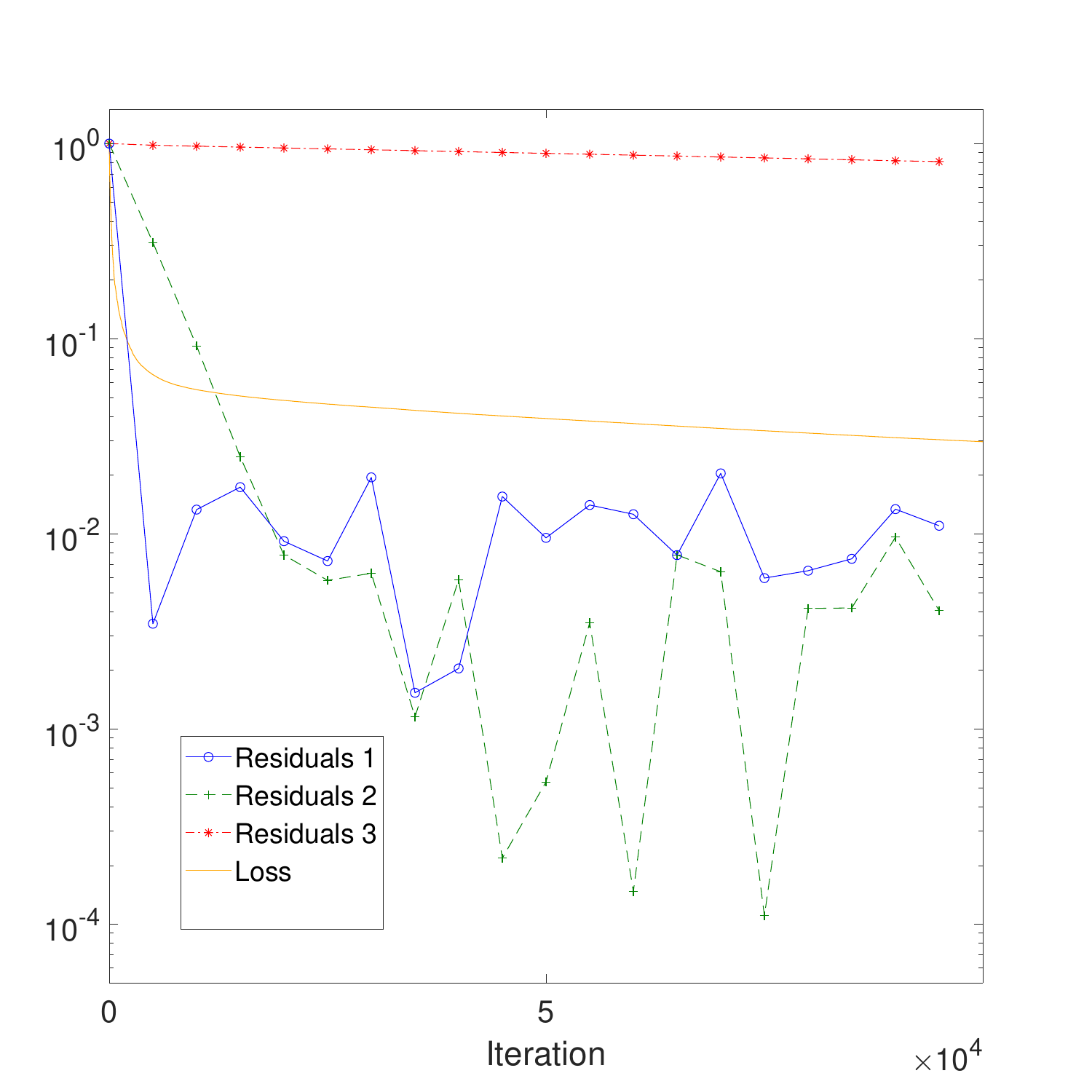}
        \captionsetup{labelformat=empty}
        \caption{Figure 1a}
        \label{Figure1}
    \end{subfigure}
    \hspace{0.05\textwidth} 
    \begin{subfigure}{0.45\textwidth}
        \centering
        \includegraphics[width=\textwidth]{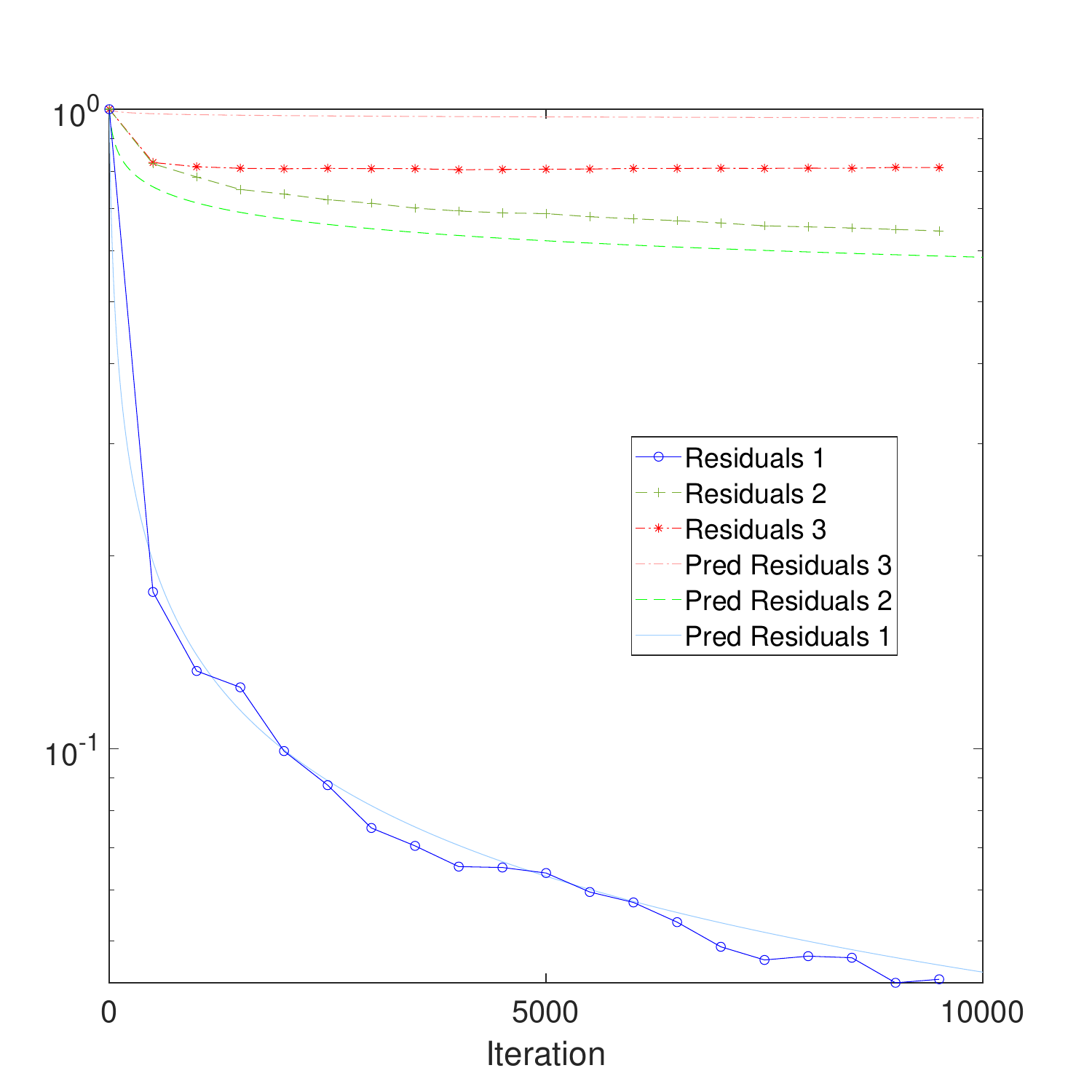}
        \captionsetup{labelformat=empty}
        \caption{Figure 1b}
        \label{Figure2}
    \end{subfigure}
\end{figure}

\begin{figure}[h]
    \centering
    \begin{subfigure}{0.45\textwidth}
        \centering
        \includegraphics[width=\textwidth]{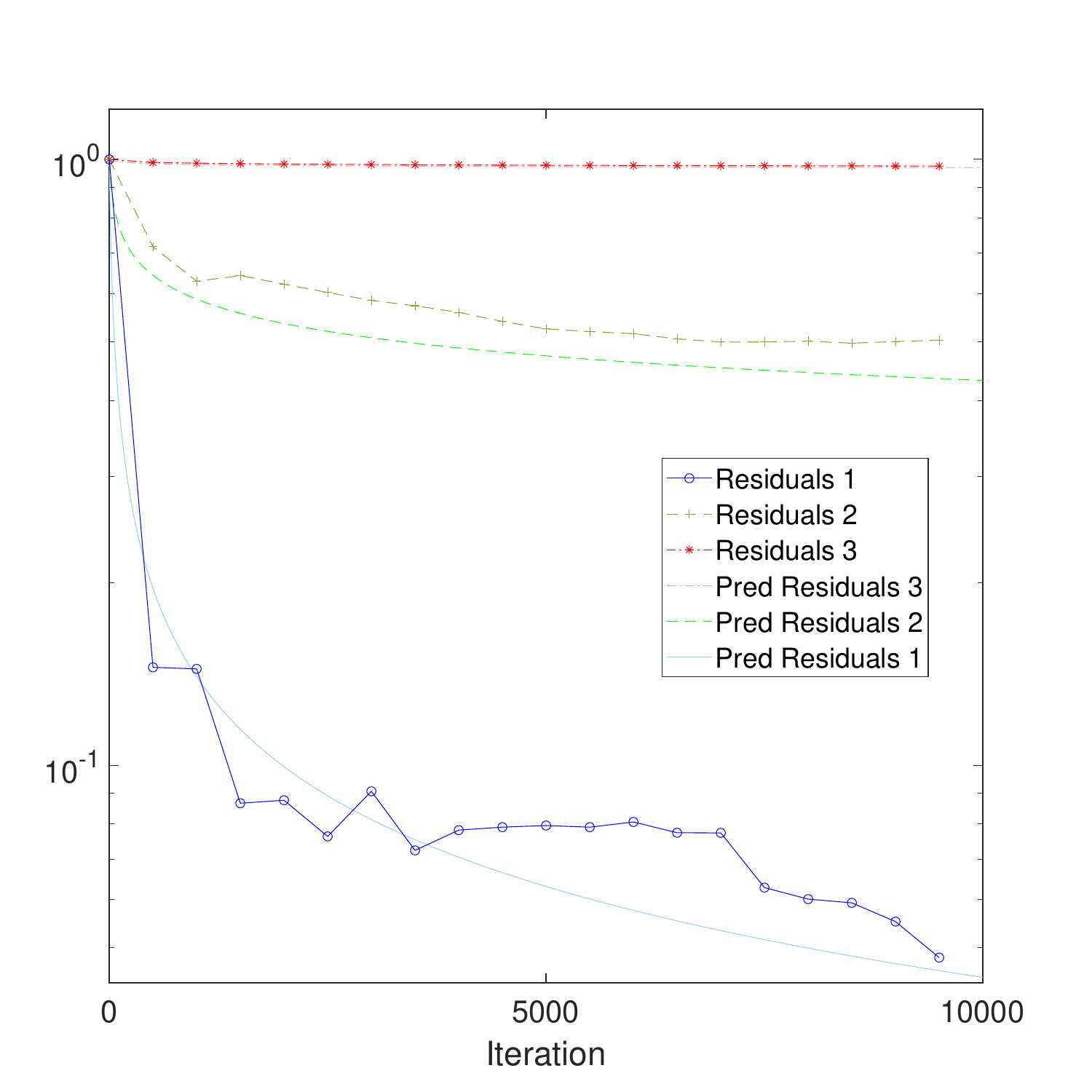}
        \captionsetup{labelformat=empty}
        \caption{Figure 2a}
        \label{Figure3}
    \end{subfigure}
    \hspace{0.05\textwidth} 
    \begin{subfigure}{0.45\textwidth}
        \centering
        \includegraphics[width=\textwidth]{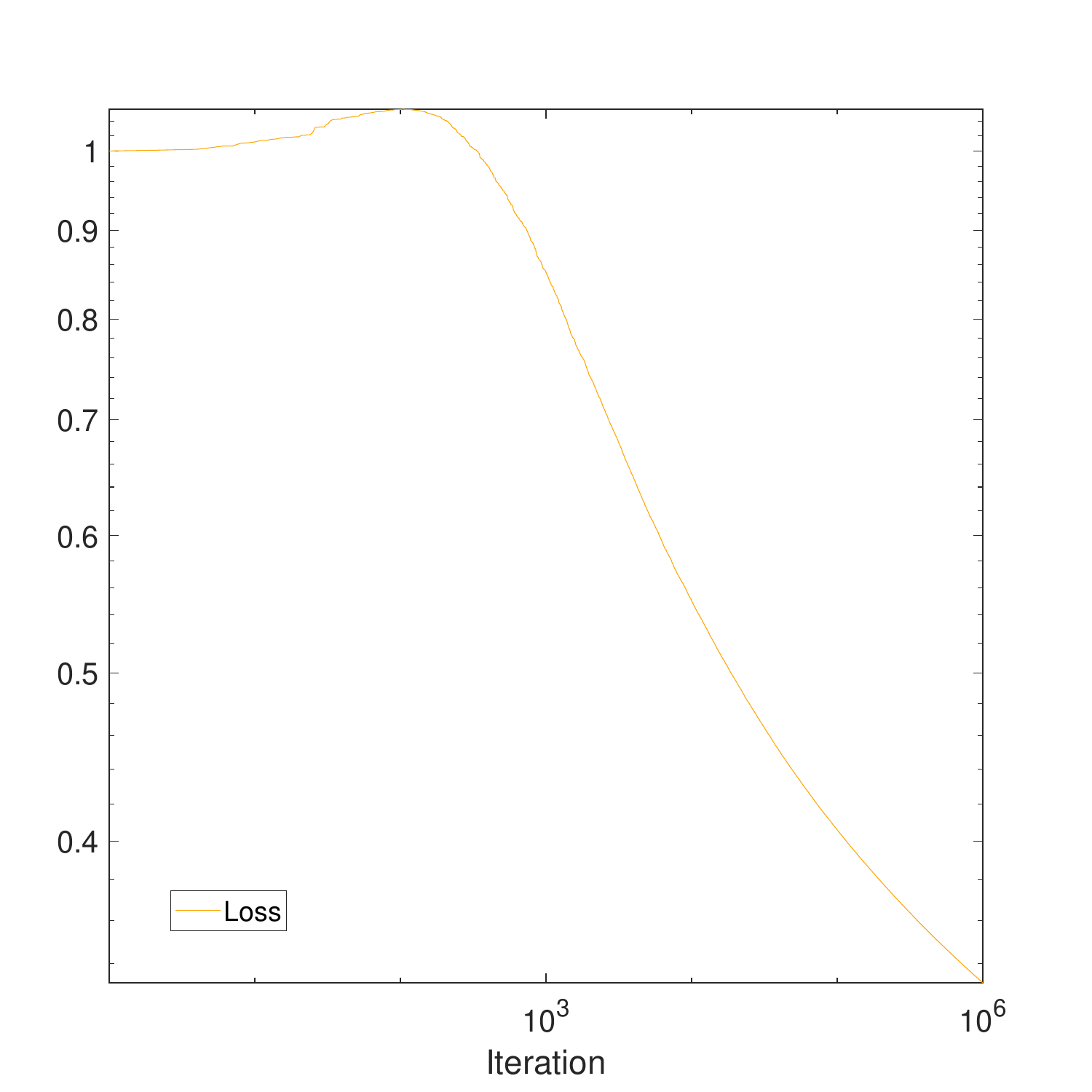}
        \captionsetup{labelformat=empty}
        \caption{Figure 2b}
        \label{Figure4}
    \end{subfigure}
\end{figure}

\begin{figure}[h]
    \centering
    \begin{subfigure}{0.45\textwidth}
        \centering
        \includegraphics[width=\textwidth]{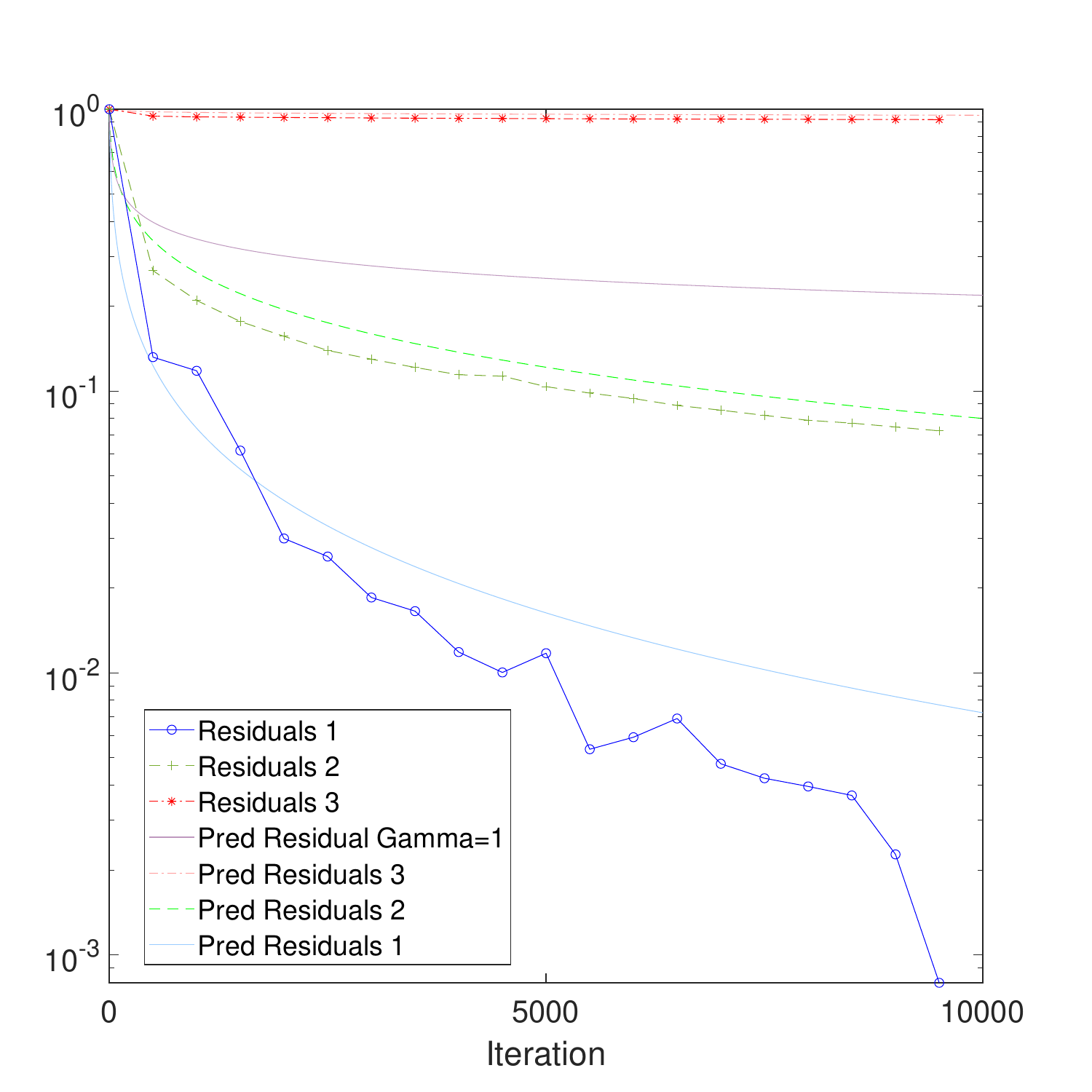}
        \captionsetup{labelformat=empty}
        \caption{Figure 3a}
        \label{Figure5}
    \end{subfigure}
    \hspace{0.05\textwidth} 
    \begin{subfigure}{0.45\textwidth}
        \centering
        \includegraphics[width=\textwidth]{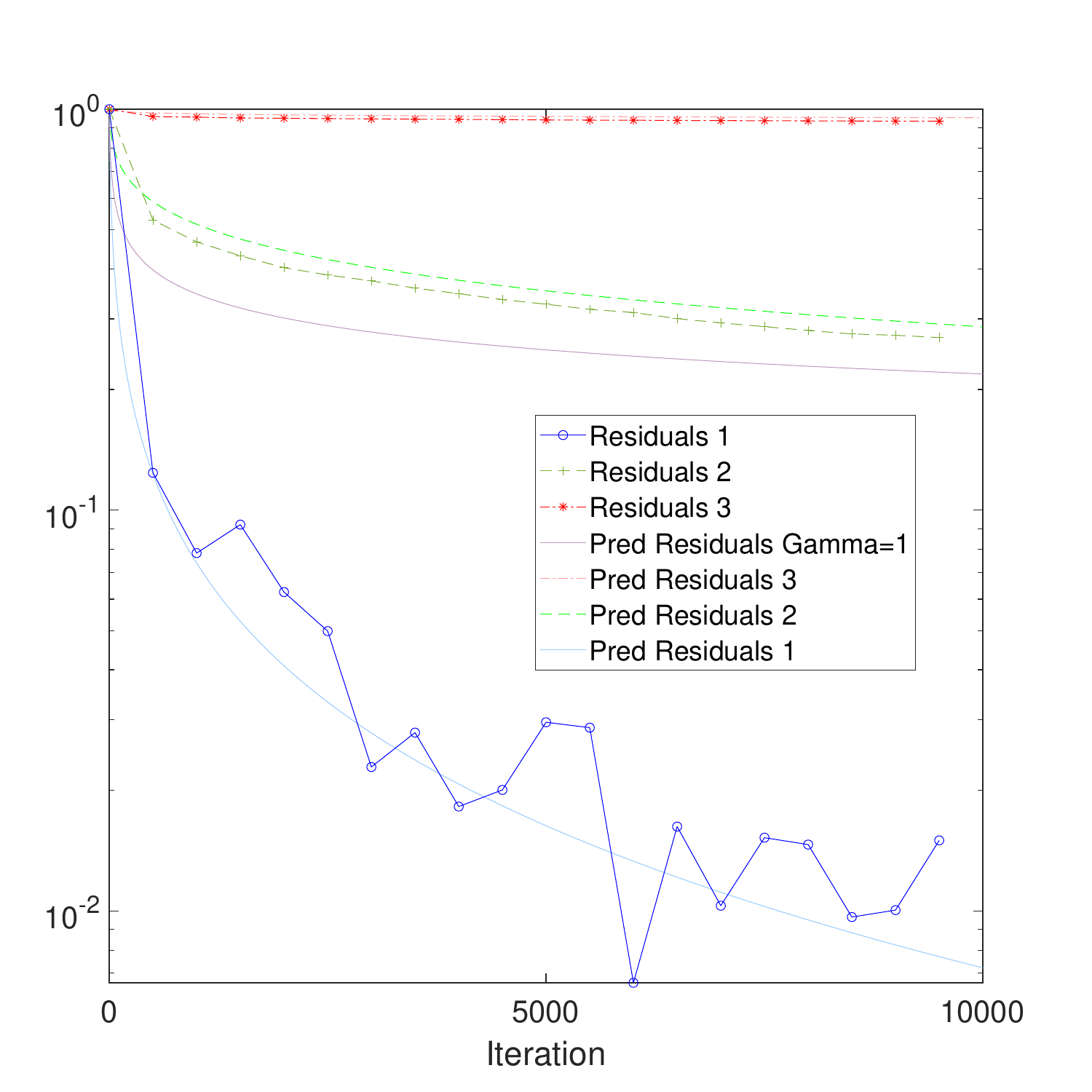}
        \captionsetup{labelformat=empty}
        \caption{Figure 3b}
        \label{Figure6}
    \end{subfigure}
\end{figure}

\begin{figure}[h]
    \centering
    \begin{subfigure}{0.45\textwidth}
        \centering
        \includegraphics[width=\textwidth]{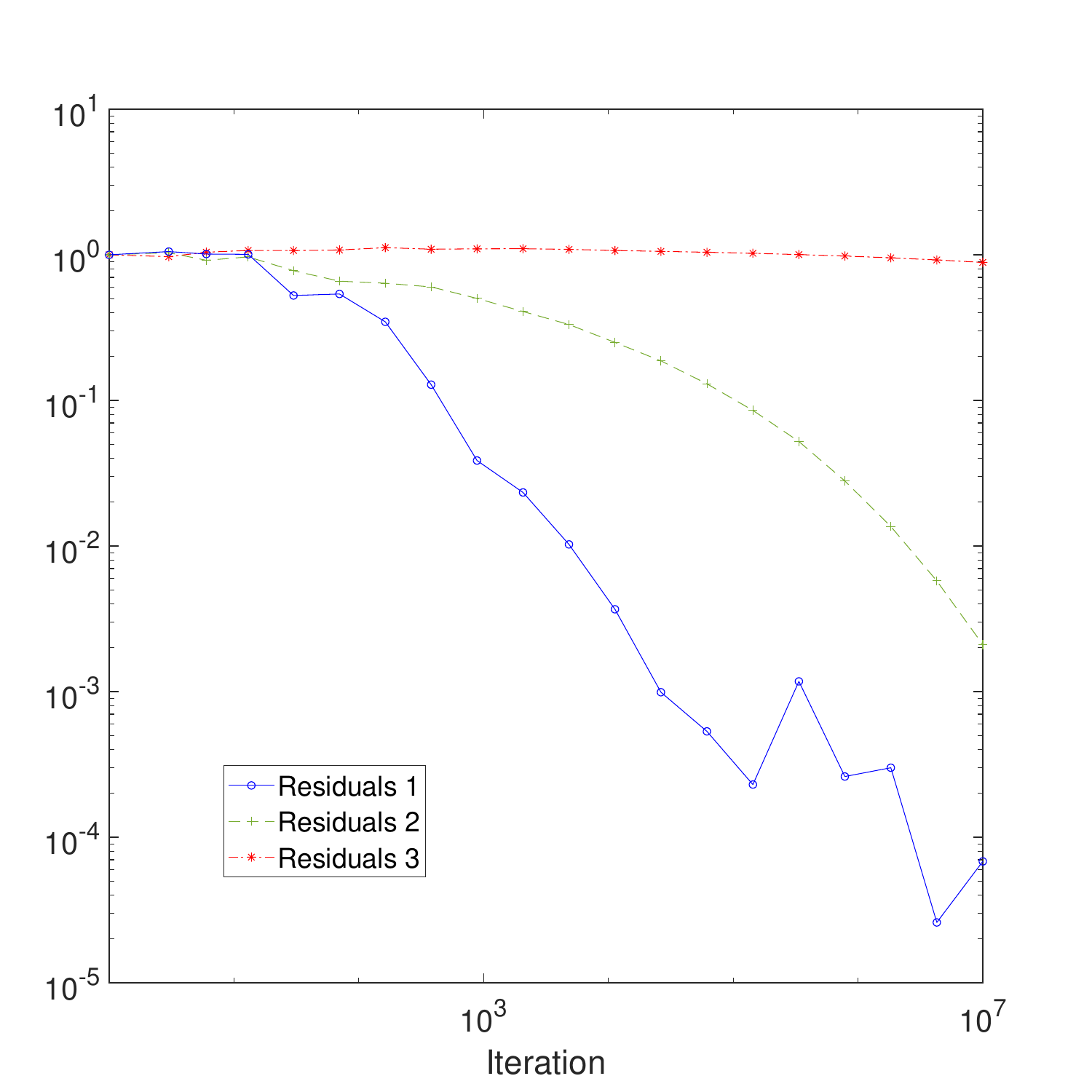}
        \captionsetup{labelformat=empty}
        \caption{Figure 4a}
        \label{Figure7}
    \end{subfigure}
    \hspace{0.05\textwidth} 
    \begin{subfigure}{0.45\textwidth}
        \centering
        \includegraphics[width=\textwidth]{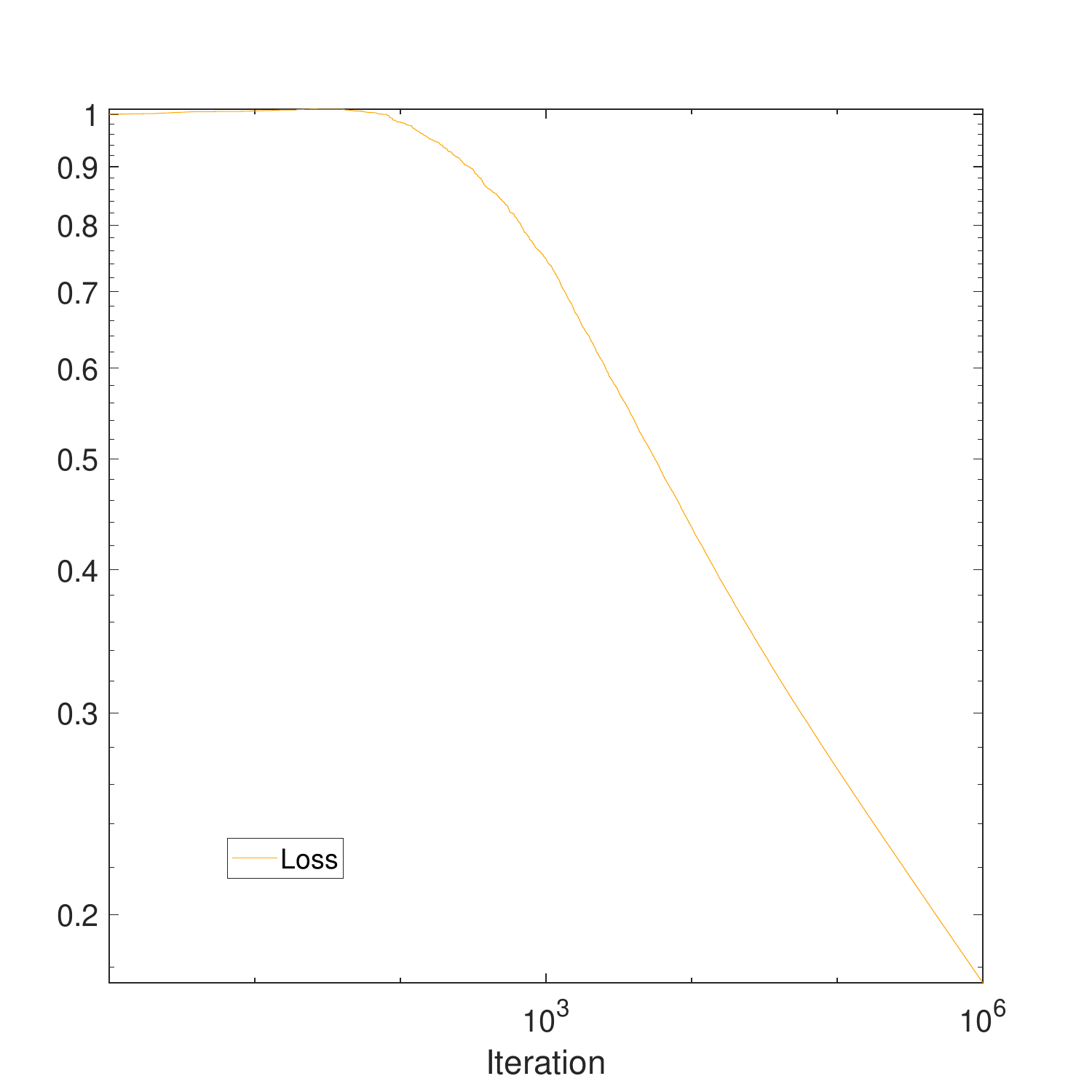}
        \captionsetup{labelformat=empty}
        \caption{Figure 4b}
        \label{Figure8}
    \end{subfigure}
\end{figure}

\newpage
\bibliographystyle{abbrv}
\bibliography{bib2}

\end{document}